\documentclass[10pt,reqno]{amsart}
\usepackage{latexsym, amsmath, amssymb}

\usepackage{bm}
\usepackage{graphics,epsfig,graphicx,graphics}
\usepackage{color, xcolor}
\usepackage{tikz}
\usetikzlibrary{patterns}
\usepackage[english]{babel}

\usepackage{hyperref}
\newtheorem{thm}{Theorem}[section]

\newtheorem{lem}[thm]{Lemma}
\newtheorem{prop}[thm]{Proposition}

\newtheorem{rem}[thm]{Remark}

\newcommand{\RR}{\mathbb{R}}
\newcommand{\rn}{\RR^N}

\newcommand{\bho}{B_{H_0}}
\newcommand{\bhor}[1]{B_{H_0}({#1})}
\newcommand{\bhorx}[2]{B_{H_0}({#2},{#1})}

\newcommand{\diver}{{\rm div}}
\newcommand{\haus}{\mathcal{H}^{N-1}}
\newcommand{\Hcap}{{\rm Cap_H}}
\newcommand{\Hcal}{{\sf M}_H}

\newcommand{\KK}{\mathcal{K}^N}

\newcommand{\la}{\lambda}
\newcommand{\nxi}{\nabla_\xi}
\newcommand{\Om}{\Omega}
\newcommand{\pa}{\partial}
\newcommand{\Sbb}{\mathbb{S}}
\newcommand{\sdue}[1]{S^2_{ij}(#1)}
\newcommand{\sign}{{\rm sign}}
\newcommand{\tr}{{\sf{Tr}}}
\newcommand{\W}[1]{\nxi^2 V(Dv(#1)) \, D^2 v(#1)}
\begin{document}
\title{An overdetermined problem for the anisotropic capacity}

\author[C. Bianchini]{Chiara Bianchini}
\author[G. Ciraolo]{Giulio Ciraolo}
\author[P. Salani]{Paolo Salani}

\address{C. Bianchini, Dip.to di Matematica e Informatica ``U. Dini'', Universit\`a degli Studi di Firenze, Viale Morgagni 67/A, 50134 Firenze - Italy}
\email{cbianchini@math.unifi.it}

\address{P. Salani, Dip.to di Matematica e Informatica ``U. Dini'', Universit\`a degli Studi di Firenze, Viale Morgagni 67/A, 50134 Firenze - Italy}
\email{paolo.salani@unifi.it}

\address{G. Ciraolo,  Dip.to di Matematica e Informatica, Universit\`a di Palermo, Via Archirafi 34, 90123, Palermo - Italy}
\email{giulio.ciraolo@unipa.it}


\maketitle

\begin{abstract}
We consider an overdetermined  problem for the Finsler Laplacian in the exterior of a convex domain in $\rn$, establishing a symmetry result for the anisotropic capacitary potential. Our result extends the one of W. Reichel [Arch. Rational Mech. Anal. 137 (1997)], where the usual Newtonian capacity is considered, giving rise to an overdetermined problem for the standard Laplace equation. Here, we replace the usual Euclidean norm of the gradient with an arbitrary norm $H$.  The resulting symmetry of the solution is that of the so-called Wulff shape (a ball in the dual norm $H_0$).
\end{abstract}




\section{Introduction}

The \emph{Newtonian capacity} of a bounded open set $\Omega$ in $\rn$, $N\geq 3$, is defined as 
\begin{equation}\label{cap}
{\rm Cap}(\Omega)=\inf \left\{\int_{\rn} \frac12 |Dv|^2 dx\ :\ v\in C^{\infty}_0(\rn),\,\,v\ge 1\,\text{ in }\Omega\right\}\,,
\end{equation}
where $Dv$ is the gradient of the function $v$ and $|\cdot|$ denotes the Euclidean norm in $\rn$.

When $N=3$, ${\rm Cap}(\Omega)$ represents the capacitance (i.e. ability to hold electrical charge) of the condenser $\Omega$ immersed in an isotropic dielectric, that is the total charge $\Omega$ can hold while maintaining a given potential energy (computed with respect to an idealized ground at infinity).

When $\Omega$ is a sufficiently smooth domain, the capacity problem \eqref{cap} admits a unique minimizer.
In fact, since Laplace equation is the Euler equation of the involved functional, this minimum problem is completely equivalent to the following Dirichlet problem
\begin{equation}\label{pbcap1}
\begin{cases}
\Delta u=0 &\qquad\text{in }\rn\setminus\overline{\Omega} \,,\\
u=1 &\qquad\text{on }\partial\Omega \,, \\
u\to 0&\qquad\text{if }|x|\to\infty \,.
\end{cases}
\end{equation}
Here, the function $u$ represents the electrostatic potential and one can ask whether there exists a set $\Omega$ such that the intensity of the corresponding electrostatic field $Du$ is constant on its boundary. This is equivalent to couple problem \eqref{pbcap1} with the extra condition
\begin{equation} \label{1ter}
|D u|=C \quad \textmd{on } \partial\Omega \,.
\end{equation} 

Since both Dirichlet and Neumann boundary conditions are imposed, the resulting problem \eqref{pbcap1}-\eqref{1ter} is {\em overdetermined} and then, in general, it is not well-posed and a solution does not exist, unless the domain $\Omega$ satisfies some additional symmetry property.
And indeed in \cite{Re} Reichel proved that \eqref{pbcap1}-\eqref{1ter} admits a solution if and only if  $\Omega$ is a ball. In other words, Euclidean balls are the only electrical conductors  such that (when embedded in an isotropic dielectric) the intensity of the corresponding electrostatic field is constant on the boundary. 

The technique used in \cite{Re} is the well-known {\em moving plane method}, which goes back first to Alexandrov and then to Serrin. The latter, in the seminal paper \cite{Se}, combined the geometric argument of Alexandrov with a smart refinement of the maximum principle to study the archetypal overdetermined problem
\begin{equation}\label{serrinpb_isotrop}
\begin{cases}
\Delta u=1 &\qquad\text{in } \Omega \,,\\
u=0 &\qquad\text{on } \Omega \,,\\
|Du|=C &\qquad\text{on } \Omega \,,
\end{cases}
\end{equation}
which is related to the minimization of the functional 
$$
\int_{\Omega}\left( \frac 12|Du|^2-u \right).
$$
Serrin proved that a solution to problem \eqref{serrinpb_isotrop} exists if and only if $\Omega$ is an Euclidean ball (and hence $u$ is radially symmetric). 
\medskip

Notice that, in both problems \eqref{serrinpb_isotrop} and  \eqref{pbcap1}-\eqref{1ter}, the radial symmetry of the solution is compelled by the isotropy of the Euclidean norm and of the Laplacian.
Considering in particular problem \eqref{pbcap1}, we see that the Laplace operator reflects the linearity of the electrical conduction law, which is in turn determined by the {\em isotropy of the dielectric} and dictates 
the use of the {\em Euclidean norm} in measuring the electric field in condition \eqref{1ter}. In this paper we investigate what happens if one considers an \emph{anisotropic dielectric} background which influences the organization of electric charges and affects the measure of the intensity of the electric field.
To this aim, in problem (\ref{cap})
we replace the Euclidean norm with a {\em generic norm $H$ which reflects the anisotropy of the medium}, thus defining the \emph{$Finsler H$-capacity} as follows
\begin{equation}\label{capminpbintro}
\Hcap(\Omega)=\inf \left\{\int_{\rn} \frac12\, H(Dv)^2 dx\ :\ v\in C^{\infty}_0(\rn),\,\,v\ge 1\,\text{ in }\Omega\right\}.
\end{equation}
$\Hcap(\Omega)$ represents the anisotropic capacitance of the set $\Omega$, that is the total charge that the set $\Omega$ can hold while embedded in the considered anisotropic dielectric medium and maintaining a given potential energy with respect to an idealized ground at infinity. 

Under suitable regularity assumptions, as in the Euclidean case the capacity problem \eqref{capminpbintro} admits a unique minimizer in $W^{1,2}(\RR^N)$ and it is in fact equivalent to the following Dirichlet problem 
\begin{equation}\label{pb_capacity}
\begin{cases}
\Delta_Hu=0&\qquad\text{in }\rn\setminus\overline{\Omega},\\
u=1&\qquad\text{on }\partial\Omega,\\
u\to 0&\qquad\text{as }H_0(x)\to\infty \,,
\end{cases}
\end{equation}
where   
$$
\Delta_H u=\diver(H(Du)\nabla_\xi H(Du)) 
$$
is  the so called \emph{Finsler Laplacian} (associated to $H$) and $H_0$ is the dual norm of $H$ (see below for precise definitions and notation).
Then, as in \cite{Re}, we investigate the overdetermined problem arising when the additional constraint 
\begin{equation} \label{3bis}
H(Du)=C \quad \textmd{ on }  \partial\Omega
\end{equation} 
is imposed. The study of geometric properties and characterization of the solution of \eqref{pb_capacity}-\eqref{3bis} is in fact the main goal of the present work.

Clearly, since the associated metric is no more radially symmetric, we can not expect Euclidean balls to be solutions and  the results and techniques from \cite{Re} do not apply anymore. In particular, the classical moving plane method is no more suitable.
Indeed the shape of the set $\Omega$ and the geometry of the solution $u$ are governed by the norm $H$ and we need to use an ad hoc technique.
For this we adapt and merge the arguments of \cite{CRS} and \cite{CS}, that in turn both exploit and suitably arrange a method from \cite{BNST}.
In \cite{CRS} the authors improve the results of \cite{Re}, weakening the regularity assumptions on the set $\Omega$. In \cite{CS}, the authors 
consider the anisotropic version of the classical Serrin's problem \eqref{serrinpb_isotrop}
\begin{equation}\label{serrinpb}
\begin{cases}
\Delta_H u=1 &\qquad\text{in } \Omega \,,\\
u=0 &\qquad\text{on } \Omega \,,\\
H(Du)=C &\qquad\text{on } \Omega \,,
\end{cases}
\end{equation}
proving that, under suitable regularity assumption, a solution exists if and only if $\Omega$ has the so-called {\em Wulff shape} associated to $H$, i.e. 
it is a ball in the dual norm $H_0$. 
Here, we will prove the same symmetry property for problem \eqref{pb_capacity}-\eqref{3bis}, i.e. for the anisotropic version of \eqref{pbcap1}-\eqref{1ter}, in perfect analogy with the interplay between problems \eqref{serrinpb} and  \eqref{serrinpb_isotrop}.

%
%
%

\subsection{Main results}

Let $N\geq 3$ and  $H: \rn \to \RR$ be a norm in $\rn$, that is a nonnegative positively homogeneous convex function;
more explicitly:
\begin{itemize}
\item[(i)] $H$ is convex;
\item[(ii)] $H(\xi) \geq 0$ for $\xi \in \rn$ and $H(\xi)=0$ if and only if $\xi=0$;
\item[(iii)] $H(t\xi) = |t| H(\xi)$ for $\xi \in \rn$ and $t\in \RR$.
\end{itemize}
Then let $H_0$ be the dual norm of $H$, that is
\begin{equation}\label{defH0}
H_0(x)=\sup_{\xi\neq 0}\frac{\langle {x};{\xi}\rangle}{H(\xi)}\quad\text{ for
}x\in\rn\,.
\end{equation}
We denote by
$B_{H}(1)$ and $B_{H_0}(1)$ the unitary balls in the norm $H$ and $H_0$ respectively; in general, for $r>0$, we set
$$
B_{H}(r)=\{\xi\in\rn\,\:\,H(\xi)<r\}\,,\qquad B_{H_0}(r)=\{x\in\rn\,\:\,H_0(x)<r\}\,.
$$
We say that a set has {\em the Wulff shape of }$H$ if it is a ball in norm $H_0$.

Given a smooth function $u$, we will use $H_0$ to measure the norm of $x\in\rn$ and $H$ to measure the norm of $D u(x)$ (then $H$ endows in fact  the dual of $\rn$, that coincides however with $\rn$).
The \emph{Finsler Laplacian} (associated to $H$) of the function $u$ is given by
\begin{equation}\label{DeltaH=}
\Delta_H u = \diver \big(H(Du) \nxi H(Du)\big).
\end{equation}
The Finsler Laplacian have been widely investigated in literature and goes back to Wulff \cite{Wu}, who considered it to describe the theory of crystals. Many other authors developed the related theory in several settings, considering both analytic aspects (see \cite{FK, CFV, CFV2, GdP1, GdP2, GdP3, WX1, WX2, CS}) and geometric points of view (see \cite{HLMG, EFT, FMP}).

In this paper we will study the anisotropic capacity problem (\ref{pb_capacity}) and the associated overdetermined problem \eqref{pb_capacity}-\eqref{3bis}. In particular, our main result is the following.

\begin{thm}\label{theorem Omega ball}
Let $\Omega$ be a bounded convex domain of class $C^{2,\alpha}$.  Let $H\in C^{2,\alpha}$ in $\rn\setminus\{O\}$ be a norm in $\rn$ such that $H^2$ is uniformly convex.
\par
Problem \eqref{pb_capacity}-\eqref{3bis} admits a solution $u$ if and only if $\Omega$ has the Wulff shape of $H$, i.e. $\Omega=B_{H_0}(r)$ for some $r>0$ (up to a translation) and $u$ is given by 
\begin{equation} \label{5bis}
u(x) = \left( \frac{H_0(x)}{r} \right)^{2-N}, \quad x \in \rn \setminus \Om \,.
\end{equation}
\end{thm}
We remark that in our assumptions on $H$ the solution of problem \eqref{pb_capacity} turns out to be classical, as shown in Theorem \ref{thm properties u} below.

Notice that the value of the constant $C$ in \eqref{3bis} must be suitably related to the geometry of the set $\Omega$ and a direct calculation (see the Appendix \ref{appendixC}) gives
\begin{equation}\label{C=}
C= \frac{(N-2)}N \frac{P_H(\Om)}{|\Om|},
\end{equation}
where $P_H$ indicates the so called anisotropic perimeter (see (\ref{anis_surf_energy}) below for its definition).

The proof of Theorem \ref{theorem Omega ball} is based on integral identities and a pointwise inequality (in the same spirit of \cite{CS} and \cite{CRS}) and can be summarized as follows. We introduce an auxiliary function $v=u^{\frac 2{N-2}}$ and prove that $v$ is quadratic in the norm $H_0$. Indeed,
by using an anisotropic version of the Minkowski inequality (see Proposition \ref{mink-ineq-prop}) and the characterization of the equality case in a generalized Newton inequality, we obtain that $\Omega$ has constant anisotropic mean curvature, and then the anisotropic Aleksandrov Theorem (Proposition \ref{AleksandrovThm}) guarantees that $\Omega$ is a ball in the suitable anisotropic metric.

\medskip

In order to apply our strategy we need several preliminary results.

First we show that if $\Om$ is a bounded convex domain with boundary of class $C^{2,\alpha}$ and $H$ is a norm of class $C^{2,\alpha}(\RR^N \setminus \{O\})$ with $H^2$ strictly convex, then Problem (\ref{pb_capacity}) admits a unique solution $u$, with $u \in C^{2}(\RR^N \setminus \Om)$ (see Theorem \ref{thm properties u}).
The proof is based on the fact that the differential problem (\ref{pb_capacity}) is the Euler equation of the minimum problem (\ref{capminpbintro}), which involves a strictly convex and differentiable functional.

Several tools from convex geometry are also needed. In particular, we will extensively use mixed volumes and mixed area measures to prove an anisotropic Aleksandrov-Fenchel inequality.

In Theorem \ref{AleksandrovThm} we also prove the anisotropic version of Aleksandrov Theorem. This result was already available in literature (see \cite{HLMG} and \cite{HeLi}). However, since our approach is simple and is in the same spirit of the proof of Theorem \ref{theorem Omega ball}, we prefer to include the proof in the paper.

\medskip

The paper is organized as follows. Section \ref{section_preliminaries} is devoted to recall and prove some preliminary results which will be useful in the proof of Theorem \ref{theorem Omega ball}, in particular we recall some well-known properties of norms in $\rn$, facts on Finsler metrics and Finsler laplacian, prove some results on Finsler capacity and properties of the corresponding capacitary function. Section \ref{section_preliminaries} is completed by recalling the definition and some basic properties of elementary symmetric functions of a matrix and tools from convex geometry and proving a Minkowski type inequality which will be crucial in the proof of Theorem \ref{theorem Omega ball}.
Section \ref{sec_mainresult} is devoted to prove Theorem \ref{theorem Omega ball}.
In Appendix \ref{appendixC} we prove \eqref{C=} and in Appendix \ref{appendixB} we give a proof of the anisotropic version of Aleksandrov Theorem.

\section{Preliminaries} \label{section_preliminaries}

\subsection{{Notations}}
For a subset $\Omega$ of $\rn$ we denote by $|\Omega|$ its volume, and by $\haus(\pa\Omega)$ the $(N-1)$-dimensional Hausdorff measure of $\pa\Omega$, that is its Euclidean perimeter, so that:
$$
|\Omega|=\int_{\Omega}d\mathcal{H}^{N}(x),\qquad \haus(\pa\Omega)=\int_{\pa\Omega}d\haus(x).
$$

For a convex set $\Omega\subset\rn$ the \emph{support function} $h(\Om,\cdot):\rn\to\RR$ of $\Omega$ is defined by
$$
h(\Om,u)=\sup_{y\in\Omega}\langle y;u\rangle,
$$
beeing $\langle \cdot; \cdot \rangle$ the standard scalar product.
The function $h$ is convex and $1$-homogeneous. Sometimes it is useful to consider its restriction to the $(N-1)$-dimensional unit sphere in $\rn$, which we denote by $\Sbb^{N-1}$.
Notice that we will indicate by $h_i, h_{ij}$ the derivatives of $h$ with respect to the $i$-th, $j$-th components of the variable $x\in\rn$.


The gradient  of a function $u:\Omega\to\rn$, evaluated at $x \in \Omega$, is the element $Du(x)$ of the dual space of $\rn$, also identified with $\rn$, which associates to any vector $y\in\rn$ the number $\langle y; Du(x)\rangle$. 
Unless otherwise stated, we will use the variable $x$ to denote a point in the ambient space $\RR^N$ and $\xi$ for an element in the dual space. The symbols $D$ and $\nxi$ will denote the gradients with respect to the $x$ and $\xi$ variables, respectively.

Accordingly, if the dual space of $\rn$ is equipped with the norm $H$, then $\rn$ turns out to be endowed with the dual norm $H_0$ given by \eqref{defH0}.


Given a  convex set $\Omega$, we denote by $\nu=(\nu^1,\ldots,\nu^N)$ its outer unit normal vector. Moreover $\nu_j=(\nu_j^1,\ldots,\nu_j^N)$ will indicates the vector of derivatives with respect to the variable $x_j$.

Einstein summation convention is in use throughout the paper.










\subsection{Norms in $\rn$} \label{subsec_norm}
Let $H: \rn \to \RR$ be a norm in $\rn$, that is
\begin{itemize}
\item[(i)] $H$ is convex;
\item[(ii)] $H(\xi) \geq 0$ for $\xi \in \rn$ and $H(\xi)=0$ if and only if $\xi=0$;
\item[(iii)] $H(t\xi) = |t| H(\xi)$ for $\xi \in \rn$ and $t\in \RR$.
\end{itemize}

The dual norm $H_0$ is defined by (\ref{defH0}).
Analogously, we can define $H$ in terms of $H_0$ as
\begin{equation*}
H(\xi) = \sup_{x\neq 0} \frac{\langle x ; \xi\rangle}{H_0(x)},\quad \xi \in \rn.
\end{equation*}
Notice that $H$ results to be the support function of the unitary ball $B_{H_0}(1)$ of $H_0$ and, in turn, $H_0$ is the support function of $B_H$,
that is
\begin{equation}\label{HhB}
H(\xi)=h(B_{H_0},\xi)\,\,\text{for }\xi\in\rn\,,\qquad H_0(x)=h(B_H,x)\,\,\text{for }x\in\rn\,,
\end{equation}
and  the convex sets $B_{H_0}$ and $B_H$ are polar of each other.

From \cite[Corollary 1.7.3]{Sc}, we have that
$H_0 \in C^1(\rn \setminus \{0\})$ if and only $B_{H}(1)$  is strictly convex. Moreover, we notice that if $H\in C^2(\RR^N\setminus\{0\})$ and $B_{H}(1)$ is uniformly convex (i.e. $H^2\in C^2_+(\RR^N\setminus\{0\})$), then the same holds for $H_0$ and $B_{H_0}(1)$. 

Since all norms in $\rn$ are equivalent, there exist positive constants $\sigma$ and $\gamma$ such that
\begin{equation} \label{norms equiv}
\sigma |\xi| \leq H(\xi) \leq \gamma |\xi|,\quad \xi \in \rn.
\end{equation}

Let $H \in C^1(\rn \setminus \{0\})$, by (iii)  we have 
\begin{equation*}
\nxi H(t\xi) = \sign(t) \nxi H(\xi), \quad \xi \neq 0,\  t\neq 0,
\end{equation*}
and
\begin{equation}\label{nablaH.xi}
 \langle\nxi H(\xi);\xi\rangle = H(\xi),\quad \xi \in \rn,
\end{equation}
where the left hand side is taken to be $0$ when $\xi = 0$. 
If $H \in C^2(\rn \setminus \{0\})$, then
\begin{equation*}
\nxi^2 H(t\xi) = \frac{1}{|t|} \nxi^2 H(\xi), \quad \xi \neq 0,\  t\neq 0,
\end{equation*}
where $\nxi^2$ is the Hessian operator with respect to the $\xi$ variable.
Hence, (\ref{nablaH.xi}) implies that
\begin{equation}\label{D2Hxi}
H_{\xi_i\xi_k}\xi_i=0 \,,
\end{equation}
for every $k=1,...,N$.

The following properties hold provided $H \in C^1(\rn \setminus \{0\})$ and $\bho$ is strictly convex (see \cite[Section 3.1]{CS}):
\begin{eqnarray}
&& H_0(\nxi H(\xi)) = 1, \quad \xi \in \rn\setminus \{0\}, \label{H0H=1} \\
&& H(D H_0(x)) = 1, \quad x \in \rn\setminus \{0\}; \label{HH0=1}
\end{eqnarray}
furthermore, the map $H \nxi H $ is invertible with
\begin{equation}\label{HnablaH inverse}
H\nxi H = (H_0 \nxi H_0)^{-1}.
\end{equation}
From \eqref{H0H=1} and the homogeneity of $H_0$, \eqref{HnablaH inverse} is equivalent to 
$$
H(\xi) D_\eta H_0(\nabla_\xi H(\xi)) = \xi \,.
$$
When $H$ and $H_0$ are of class $C^2(\RR^N\setminus\{0\})$, by differentiating this expression and using \eqref{nablaH.xi} and \eqref{D2Hxi}, we obtain
\begin{equation} \label{miracolo}
\nabla_\xi^2 V D_\eta^2 V_0 (\nabla_\xi H) = Id \,,
\end{equation}
where $V=H^2/2$ and $V_0=H_0^2/2$.

\subsection{Finsler metric} \label{subsec_metric}

Level sets of the norms $H$ or $H_0$ have a special role in the study of the anisotropic space, as well as Euclidean balls have in the Euclidean space.
More precisely we will say that a set $E$ is \emph{Wulff shape} of $H$ if there exist $t>0$ and $x_0\in\rn$ such that
$$
E=\{x\in\rn\ :\ H_0(x-x_0)\le t\};
$$
in other words, if it is a $H_0$-ball.
The set $E$ is then  denoted by $\bhorx{t}{x_0}$ where $x_0$ is the center and $t$ is the $H_0$-radius of the ball. When $x_0=0$, we simply write $\bhor{t}$ for $\bhorx{t}{0}$.

Notice that the unitary $H_0$-ball can be seen as the image of the function $\Phi:S^{N-1}\to\rn$ such that $\Phi(\xi)=\nabla_{\xi}H(\xi)$, thanks to the properties of the norm $H$ and its dual $H_0$.

For a sufficiently regular set $\Om\subset \rn$ we denote by $P_H(\Om)$ its \emph{anisotropic perimeter}, or anisotropic surface energy, that is
\begin{equation}\label{anis_surf_energy}
P_H(\Om)=\int_{\pa \Om}H(\nu) d\haus(x).
\end{equation}
Obviously,  when $H$ is the Euclidean norm, then $P_H(\Om)$ is the usual perimeter of $\Omega$.

Following \cite{WX2}, the \emph{anisotropic mean curvature} of $\partial \Omega$, which we shall denote by $\Hcal$, is defined by
\begin{equation} \label{MHOm}  
\Hcal(\Omega)= H_{\xi_i\xi_j}\nu_i^j \,.
\end{equation}
The anisotropic mean curvature arises for instance when one considers the anisotropic surface energy (\ref{anis_surf_energy}) of a hypersurface $\partial \Omega$,
so that $H(\nu)$ gives the unit energy per unit area of a surface element having normal $\nu$.
Since,
$$
|\Omega| = \frac{1}{N} \int_{\pa\Omega} \langle x ; \nu\rangle d \haus,
$$
if one considers the critical points of the shape operator $P_H(\cdot)$ for volume-preserving variations, then one obtains that they satisfy $\Hcal=constant$. 
We notice that if $H(\xi)=|\xi|$ then $\Hcal$ is the usual mean curvature normalized so that for the Euclidean unit sphere $B$ it holds $\Hcal(B)=(N-1)$. 

As it is well known, in the Euclidean setting the only compact constant mean curvature hypersurfaces without boundary are Euclidean balls (Aleksandrov's Theorem). In the Finsler metric an analogous result holds.
\begin{thm}[Anisotropic Aleksandrov's Theorem]\label{AleksandrovThm}
Let $H$ be a norm of $\rn$ of class $C^2(\rn \setminus \{O\})$ such that $H^2$ is uniformly convex, and let $\pa \Omega$ be a compact hypersurface without boundary embedded in Euclidean space of class $C^2$.
If $\Hcal(x)$ is constant for every $x\in \pa \Omega$ then $\Omega$ has the Wulff shape of $H$.
\end{thm}
A proof of the previous result can be found in  \cite{HLMG} and in \cite{HeLi}. In Appendix B we present an alternative proof which is more in the spirit of Reilly's proof \cite{Re} and of our proof of Theorem \ref{theorem Omega ball}.


\subsection{Finsler Laplacian}\label{section-finslerlaplacian}

The Finsler Laplacian associated to a norm $H$ is the operator $\Delta_H$ defined by
$$
\Delta_H u=(H(Du)H_{\xi_i\xi_j}(Du)+H_{\xi_i}(Du)H_{\xi_j}(Du))\;u_{ij}.
$$
This operator extends the notion of Laplacian to the anisotropic space $\rn$ endowed with a generic norm $H$.
The classical Laplacian corresponds to $\Delta_H$ in the case $H$ is the Euclidean norm.

Notice that, thanks to the regularity and the homogeneity properties of the norm $H$, the Finsler Laplacian is a strictly elliptic operator; indeed
$$
(H(\xi) H_{\xi_i\xi_j}(\xi)+H_{\xi_i}(\xi)H_{\xi_j}(\xi))\xi_i\xi_j=H^2(\xi)\ge C |\xi|^2\,,
$$
where $C=\min\{H(\xi)\,:\,|\xi|=1\}$.

Several results, which are valid in the Euclidean case, hold true in the anisotropic case too; we here present only few of them.

Let $\alpha_N$ be the perimeter of the unit ball with respect to $H_0$. We refer to
\begin{equation}\label{Gamma def}
\Gamma(x) = \frac{H_0^{2-n}(x)}{\alpha_n (n-2)}
\end{equation}
as the fundamental solution of the Finsler Laplacian in $\rn$, $N \geq 3$, since $\Gamma$ solves
\begin{equation*}
-\Delta_H \Gamma = \delta_0,
\end{equation*}
where $\delta_0$ is the Dirac measure centered at the origin (see \cite{FK}). 

\begin{prop}[Weak Comparison Principle \cite{FK}]
Let $E$ be a bounded domain and assume that
\begin{equation*}
-\Delta_H u \leq -\Delta_H v \quad \textmd{in } E , \quad \textmd{and } u \leq v \ \textmd{ on } \pa E,
\end{equation*}
then
\begin{equation*}
u \leq v \quad \textmd{a.e. in } E.
\end{equation*}
\end{prop}

In particular, the following maximum principle holds.
\begin{prop}[Maximum Principle \cite{FK}]
If $\Delta_H u = 0$ in $E$, then 
$$
\min_{\pa E} u \leq u(x) \leq \max_{\pa E} u,
$$
almost everywhere in $E$.
\end{prop}

An analogous of the mean curvature formula for the Laplacian has been proved in \cite{WX1} where the anisotropic mean curvature $\Hcal$ has been taken into account.
\begin{prop}[\cite{WX1}]
Let $u$ be a regular function with a regular level set $S_t=\{x\in\rn\ :\ u(x)=t\}$. 
The following expression holds at every $x\in S_t$:
\begin{equation}\label{Hlaplacian-curvatura}
\Delta_H u= \Hcal(S_t) H(Du)+H_{\xi_i}H_{\xi_j}u_{ij}.
\end{equation}
\end{prop}
\subsection{Finsler Capacity}\label{section-finsler}
We recall that the \emph{Finsler capacity} or \emph{anisotropic capacity} of a convex bounded open set $\Om \subset \rn$ is defined by
\begin{equation}\label{capminpb}
\Hcap (\Om) = \inf \left\{ \frac{1}{2} \int_{\rn} H^2(Dv) dx:\ v \in C_0^\infty (\rn),\ v_{|_\Om} \geq 1\right\}.
\end{equation}
The function $u$ such that
$$
 \frac{1}{2} \int_{\rn} H^2(Du) dx = \Hcap(\Om)
$$ 
is called the $H$-\emph{capacitary potential}  of $\Om$ and it satisfies Problem \eqref{pb_capacity}, as we will show in Theorem \ref{thm properties u}.

The notion of capacity can be extended to the so called \emph{relative capacity}: the Finsler capacity of a convex bounded open set $\Om \subset \rn$ with respect to a superset $E\supset\Omega$ is defined by
\begin{equation}\label{capminpbOE}
\Hcap (\Om;E) = \inf \left\{ \frac{1}{2} \int_{E} H^2(Dv) dx:\ v \in C_0^\infty (E),\ v_{|_\Om} \geq 1\right\}.
\end{equation}

In the following, we show two prime examples which will be useful later.
Let us consider the radial case $\Om = B_{H_0}(r)$ and let $u_r$ be solution to \eqref{pb_capacity} in $\rn\setminus\overline{B_{H_0}(r)}$.
Since $\Gamma$ in \eqref{Gamma def} is the fundamental solution, it is clear that $u_r(x)=\alpha_N(N-2) r^{N-2} \Gamma(x)$, that is
\begin{equation}\label{v r}
u_r(x)= \frac{H_0^{2-N}(x)}{r^{2-N}}, \quad x \in \rn \setminus B_{H_0}(r).
\end{equation}
Moreover, we have that
\begin{equation}\label{H D vr}
H(Du_r(x)) = \frac{N-2}{r}, \quad \textmd{for } x \in \pa B_{H_0}(r).
\end{equation}

Another crucial example is the annular ring case, where the Finsler capacity of $B_{H_0}(r_1)$ with respect to $B_{H_0}(r_2)$ is considered, for $0< r_1 < r_2$.
The function
\begin{equation}\label{v R1 R2}
u_{r_1,r_2}(x) = \frac{H_0^{2-N}(x)-r_2^{2-N}}{r_1^{2-N} - r_2^{2-N}}
\end{equation}
minimizes Problem (\ref{capminpbOE}) and it solves the capacity problem in the ring
\begin{equation*}
\begin{cases}
\Delta_H u = 0, & \textmd{in } B_{H_0}(r_2) \setminus \overline{B}_{H_0}(r_1),\\
u=1, & \textmd{if } {H_0}(x)=r_1, \\
u=0, & \textmd{if } H_0(x)=r_2.
\end{cases}
\end{equation*}

In the following theorem we prove that Problem (\ref{capminpb}) for the Finsler Capacity is equivalent to  the differential problem (\ref{pb_capacity}) and we give some crucial estimates on the $H$-capacitary function $u$.
\begin{thm} \label{thm properties u}
Let $\Om$ be a bounded convex domain with boundary of class $C^{2,\alpha}$ such that $O\in \Om$. Let $H$ be a norm of $\rn$ of class $C^{2,\alpha}(\rn \setminus \{O\})$ such that $H^2$ is uniformly convex.
There exists a unique solution $u$ to problem \eqref{pb_capacity}, $u \in C^{2}(\rn \setminus \Om)$ , and it satisfies the following properties:
\begin{itemize}
\item[(i)] $0<u<1$ in $\rn \setminus \overline{\Om}$;
\item[(ii)] there exist two positive constants $A_1$ and $A_2$ depending on $\Om$ such that
\begin{eqnarray}
A_1 \Gamma(x) \leq u(x),& \quad x \in \rn \setminus \Om,\\ \label{A Gamma leq u}
u(x) \leq A_2 \Gamma(x),& \quad x\in \rn\setminus{\bhor{R_1}},\label{u leq A Gamma}
\end{eqnarray}
where $\overline{\Omega} \subset\bhor{R_1}$ and $\Gamma$ is given by \eqref{Gamma def};
\item[(iii)] $H(Du) \neq 0$ in $\rn \setminus \Om$;
\item[(iv)] there exist positive constants $B_1,B_2$ and $B_3$ depending on $\Om$ such that
\begin{equation}\label{Du leq}
B_1 \frac{\Gamma(x)}{H_0(x)} \leq  H (Du(x)) \leq B_2 \frac{\Gamma(x)}{H_0(x)},
\end{equation}
and
\begin{equation}\label{D2u leq}
| D^2u(x)|  \leq \frac {B_3}{H_0^{N-2}(x)},
\end{equation}
for $x$ sufficiently far away from the origin.
\end{itemize}
\end{thm}

\begin{proof}
For every $r>0$, sufficiently large, let us define the function $u_r$ as the solution to the capacity problem in $\bhor{r}\setminus\overline{\Omega}$; that is
$$
\begin{cases}
\Delta_H u_r=0\qquad&\text{in }\bhor{r}\setminus\overline{\Omega},\\
u_r=1\qquad&\text{on }\pa\Omega,\\
u_r=0\qquad&\text{on }\pa\bhor{r}.
\end{cases}
$$
Notice that, by the strictly convexity of $H^2$ the function $u_r$ is the unique minimizing function for the capacity problem (\ref{capminpbOE}) related to the sets $\Omega$ and $\bhor{r}$.

Thanks to the comparison principle, if $r>s$ then $u_r(x)\ge u_s(x)$ for every $x\in\bhor{r}\setminus\overline{\Omega}$. 
Hence the function 
$$
u=\lim_{R\to\infty}u_R(x)
$$
is well defined, for $x\in\rn\setminus\overline{\Omega}$
and the sequence $u_R$ is in fact uniformly convergent. Then we are going to deduce the estimates (i)-(iv) for the functions $u_R$ and show that the involved constants do not depend on $R$, so that we will obtain the desired estimates for $u$ by passing to the limit as $R\to\infty$.

Let $0<R_0<R_1$ be given by   
$$
R_0=\sup\{r>0:\ \bhor{r} \subset \Om\}\quad\text{ and }\quad R_1=\inf\{r>0:\ \Om \subset \bhor{r} \}.
$$
Consider $u_{R_0,R}$ defined as in (\ref{v R1 R2}) in the ring $\bhor{R}\setminus\overline{\bhor{R_0}}$. 
By comparison principle it holds
\begin{equation}
u_R(x)\ge \frac{H_0^{2-N}(x)-R^{2-N}}{R_0^{2-N}-R^{2-N}},
\end{equation}
for every $x\in \bhor{R}\setminus\overline{\Omega}$ which implies
$$
u(x)\ge\frac{H_0^{2-N}(x)}{R_0^{2-N}} \,.
$$
for $x\in\rn\setminus\overline{\Omega}$ and hence inequality (\ref{A Gamma leq u}) holds.

On the other hand we can compare the function $u_R$ with $u_{R_1,R}$ defined as in (\ref{v R1 R2}) in the ring $\bhor{R}\setminus\overline{\bhor{R_1}}$ and we obtain inequality  (\ref{u leq A Gamma}) for $x\in\rn\setminus\overline{\bhor{R_1}}$, and in fact the same holds in $\RR^N\setminus \Omega$ since $u\leq 1$. 
Hence \emph{(i)} and \emph{(ii)} are proved.

Let us investigate the regularity of $u_R$. By using an argument analogous to the one used in the proof of \cite[Proposition 2.3]{CS}, we have that {$u_R\in C^{1,\alpha}(B_R\setminus\overline{\Omega})$}. Indeed the result follows from \cite[Chapter 4]{LK} and the fact that the set of points of non-differentiability of $H\nxi H$ consists of just a point, the origin.
Furthermore, by arguing as in the proof \cite[Lemma 2]{Le}, one can show that in fact $H(Du_R)$ does not vanish. 
Indeed, the proof of \cite[Lemma 2]{Le} can be adapted to our case since only the following ingredients are needed: a weak comparison principle, estimates like the ones in (i) and (ii), and the equivalence between norms which is given by \eqref{norms equiv}.
More precisely, a close inspection of the proof of \cite[Lemma 2]{Le} shows that 
$$
H(Du_R)\ge A\frac 1{H_0(x)}\ge B_1 H_0^{1-N}(x),
$$
for $x\in\bhor{R}\setminus\overline{\Omega}$, where $A, B_1$ are constants depending only on the set $\Omega$ and the dimension $N$.
Hence the first condition  in (\ref{Du leq}) holds for $u$, again by passing to the limit.

Notice that, since $Du_R$ does not vanish in $\bhor{R}\setminus\overline{\Omega}$, the differential operator $\Delta_H$ has $C^{0,\alpha}$ coefficients. 
Thanks to \cite[Theorem 6.15]{GT} we obtain $u_R\in C^{2,\alpha}(\overline{\bhor{R}\setminus{\Omega}})$.

To prove the second inequality in \eqref{Du leq}, we rescale $u$ and define
\begin{equation*}
U(y)= \rho^{N-2} u_R(\rho y),
\end{equation*}
where  $\frac{R_1}r<\rho<R$, for some $r>0$, and $y\in \bhor{1}\setminus\overline{\bhor{r}}$. 
We notice that $U$ satisfies $\Delta_H U = 0,$ in $\bhor{1}\setminus\overline{\bhor{r}}$.
From the maximum principle and using (i) and (ii) we have that $U$ is uniformly bounded in $\bhor{1}\setminus\overline{\bhor{r}}$. 

We notice that 
\begin{equation*}
\Delta_H U = \sum_{i,j=1}^n a_{ij} U_{ij},
\end{equation*}
where
\begin{equation*}
a_{ij}(x) = [H_{\xi_i}(DU(x))H_{\xi_j}(DU(x)) + H(DU(x)) H_{\xi_i \xi_j}(DU(x))] .
\end{equation*}
From the first inequality in \eqref{Du leq}, $|D U|$ is bounded away from zero; moreover, the homogeneity properties of $\nxi H$ and $\nxi^2 H$, imply that $a_{ij}$ are bounded as sum of $0$-homogeneous functions. 
Thus $U$ is solution of a uniformly elliptic quasilinear equation and from standard regularity results {\cite{To}, \cite{GT}} we obtain that $|DU|$ is uniformly bounded and $U\in C^{2}$, that is
$$
H(Du_R(\rho y))\le \frac{B_2}{\rho^{N-1}}, \qquad\text{for }y\in\bhor{1}\setminus\overline{\bhor{r}}.
$$
Let $x\in\rn\setminus\overline{\Omega}$; define $\rho=H_0(x)$ and let $R$ be sufficiently large.
Hence
$$
H(Du_R(x))\le B_2H_0^{1-N}(x),
$$
for every $x$ such that $H_0(x)\ge {R_1}/r$.
This concludes the proof of (\ref{Du leq}).

Estimate (\ref{D2u leq}) follows by Schauder's estimates \cite[Theorem 6.2]{GT} applied to $U$. Indeed there exists a constant $B_3$ depending only on the dimension, the regularity and the ellipticity constants (which are independ of $R$ for the homogeneity of the norm $H$) such that
$$
|D^2U(x)|\le B_3 \max u_R\le B_3,
$$
and hence \ref{D2u leq} holds for $u_R$ and then for $u$.

It remains to show that in fact $u$ solves the differential Problem (\ref{pb_capacity}).
As already pointed out, thanks to the strict convexity of $H^2$, the function $u_R$ is the unique minimizing function to the capacity problem in $\bhor{R}\setminus\overline{\Omega}$ (see \cite[Paragraph 8.2.3]{E}).
{Thanks to} the homogeneity and the regularity of $H$ and from the previous estimates we have that
$$
\lim_{R\to\infty}\int_{\bhor{R}\setminus\overline{\Omega}}H^2(Du_R)=\int_{\rn\setminus\overline{\Omega}}H^2(Du),
$$
which implies that $u$ solves the minimum problem (\ref{capminpb}). Moreover, since for every $\phi_R\in C^{\infty}_0(\bhor{R}\setminus\overline{\Omega})$ it holds
$$
\int_{\bhor{R}\setminus\overline{\Omega}}H(Du_R)\nxi H(Du_R)\cdot D\phi_R= 0,
$$
we deduce 
$$
\int_{\rn\setminus\overline{\Omega}}H(Du)\nxi H(Du)\cdot D\phi=0,
$$
for every $\phi\in C^{\infty}_0(\rn\setminus\overline{\Omega})$, that is $u$ is a weak solution to Problem (\ref{pb_capacity}).
Notice that, in fact, {$u$ is a classical solution} since, by using a local argument and \cite[Theorem 6.13]{GT}, we obtain that $u\in C^0({\rn\setminus\Omega})\cap C^{2,\alpha}({\rn\setminus\overline{\Omega}})$.

We can finally conclude that in fact $u\in C^{2,\alpha}(\rn\setminus\Om)$ by applying \cite[Theorem 6.19]{GT} in the set $\{u>\frac12\}\setminus\overline{\Om}$.
\end{proof}

\begin{rem}
	Notice that if $\Omega$ is assumed to be uniformly convex, then estimate (\ref{u leq A Gamma}) holds in the whole $\rn\setminus\overline{\Omega}$ as  proved in the following.
Indeed, since $\Omega$ is uniformly convex,  there exists  $R_*\geq R_1$ such that for any $y \in \pa \Om$ there exists $x^*$ such that the ball $\bhorx{R^*}{x^*}$ contains $\Om$ and is tangent to $\pa \Om$ at $y$ and $\bhor{R}\supseteq \bhorx{R_*}{x^*}$. 
Moreover $ \bhorx{R_*}{x^*}\subseteq \bhorx{3R}{x^*}$.
By considering the function $u_{R^*,3R}(x-x^*)$, the comparison principle yields that $u_R\le u_{R_*,3R}(x-x^*)$ in $\bhor{R}\setminus \overline{\bhorx{R^*}{x^*}}$.
By varying the point $y\in\pa\Omega$ and using the uniform convexity of $\Omega$, we prove that there exists a constant $A_2$, depending only on the set $\Omega$ and the dimension $N$, such that
$$
u_R(x)\le A_2{H_0^{2-N}(x)},
$$
for $x\in \bhor{R}\setminus\overline{\Omega}$ and, by passing to the limit, the 
same estimate holds for $u(x)$, $x\in\rn\setminus\overline{\Omega}$.
\end{rem}


\subsection{Elementary functions of a matrix} \label{subsec_elementary}

Given a matrix $A=(a_{ij})\in\RR^{n\times n}$, for any $k=1,\dots,n$ we denote by $S_k(A)$ the sum of all the principal  minors of $A$ of order $k$. 
In particular, $S_1(A)=\tr(A)$, the trace  of $A$, and $S_n(A)= \det (A)$, the determinant of $A$. 
More explicitly
$$
S_k(A)=\frac{1}{k!}\sum
\delta\left(\begin{array}{cc}i_1,\dots,i_k\\j_1,\dots,j_k\end{array}\right)a_{i_1j_1}\cdots a_{i_kj_k}\,,
$$
where $i_r,j_r\in\{1,\dots,k\}$ and the Kronecker symbol $\delta\left(\begin{array}{cc}i_1,\dots,i_k\\j_1,\dots,j_k\end{array}\right)$ has value $+1$ (respectively 
$-1$) when $i_r\neq i_s$ for $r\neq s$ and $(j_1,\cdots,j_k)$ is an even (respectively odd) permutation of $(i_1,\cdots,i_k)$, otherwise it has value $0$.

By setting
\begin{equation}\label{paolo3.5}
S^k_{ij}(A)=\frac{\partial}{\partial a_{ij}}S_k(A)=\frac{1}{(k-1)!}\sum_{(i_r,j_r)\neq(i,j)}
\delta\left(\begin{array}{cc}i_1,\dots,i_{k-1}\\j_1,\dots,j_{k-1}\end{array}\right)a_{i_1j_1}\cdots a_{i_{k-1}j_{k-1}}\,,
\end{equation}
we can write
\begin{equation}\label{paolo4}
S_k(A)=\frac1k\sum S^k_{ij}(A)a_{ij}\,,
\end{equation}
which is nothing more than the Euler's rule for homogeneous functions ($S_k$ being homogeneous of order $k$).
In particular, for $k=n$,  we have
\begin{equation}\label{determinant}
\det(A)=S_n(A)=\frac1n\sum S^n_{ij}(A)a_{ij}\,.
\end{equation}

We also notice that $S^n_{ij}(A)$ is the $(i,j)$-cofactor of $A$. Then \eqref{determinant} also coincides with the so called cofactor (or Laplace) expansion of the determinant; moreover, if $\det(A)\neq 0$ and we denote by $a^{ij}$ the elements of the inverse matrix $A^{-1}$ of $A$, we have
\begin{equation}\label{SNij}
S^n_{ij}(A)=\det(A)a^{ji}\,.
\end{equation}

For further use, we also notice that \eqref{determinant}, \eqref{paolo3.5} and the chain rule for derivatives yield
\begin{equation}\label{derivdet}
\frac{d}{dt}\det(A+tB)_{|t=0}=\sum S^n_{ij}(A)b_{ij}
\end{equation}
for any couple of $n\times n$ matrices $A$ and $B$.

Another case of special interest in our applications is when $k=2$. 
In this case, one has
\begin{equation}\label{eq200}
S_2(A)=\frac12\sum_{i,j}S^2_{ij}(A)a_{ij}\,,
\end{equation}
where
$$
S^2_{ij}(A)=\left\{\begin{array}{ll}
-a_{ji}\quad&\text{if }i\neq j,\\
\\
\sum_{k\neq i}a_{kk}\quad&\text{if }i=j\,.
\end{array}
\right.
$$

The next lemma shows a generalization of Newton's inequality to not necessarily symmetric matrices. 
This inequality, together with the characterization of the equality case, is in fact one of the crucial ingredients in the proof of our main result.
\begin{lem}[\cite{CS}, Lemma 3.2]\label{lemma-newton}
Let $B$ and $C$ be symmetric matrices in $\RR^{n\times n}$, and let $B$ be positive semidefinite. Set $A=BC$. Then the following inequality holds:
\begin{equation}\label{newtonIneq}
S_2{(A)}\le\frac{n-1}{2n}\tr(A)^2\, .
\end{equation}
Moreover, if ${\tr} (A)\neq 0$ and equality holds in (\ref{newtonIneq}), then 
\begin{equation*}
A=\frac{{\tr}(A)}{n}\, I \,,
\end{equation*}
and $B$ is, in fact, positive definite.
\end{lem}

Of particular interest in our approach is the quantity $S_2(W)$, where $W=\nxi^2 V(Dv)D^2v\in\RR^{N\times N}$, with $V(\xi)=\frac 12H^2(\xi)$, $H$ a $C^2(\rn\setminus\{0\})$ norm, and $v\in W^{2,2}(\Om)\cap C^{1}(\overline{\Om})$ in a bounded open set $\Om$ with $v=1$ on the $\pa \Omega$.
Notice that, in this case it holds
\begin{equation} \label{S2ijW}
\sdue{W}=
\begin{cases}
-(\nxi^2V(Dv)D^2v)_{ji} &\qquad\text{if }i\neq j,\\
-(\nxi^2V(Dv)D^2v)_{ji}+\Delta_Hv &\qquad\text{if }i=j.
\end{cases}
\end{equation}
Moreover in this setting $\sdue{W}$ is divergence free, in the following (weak) sense (see [\cite{CS}, (4.14)])
\begin{equation}\label{divS2=0}
\frac{\pa}{\pa x_j}\sdue{W}=0.
\end{equation}


\subsection{Ingredients of convex geometry} \label{subsec_convex_geometry}

We briefly present some notions and results of convex analysis. For their proof and additional details we refer to \cite{Sc}.

We indicate by $\KK$ the set of convex subsets of $\rn$.
For $K,L$ in $\KK$ we define their Minkowski sum as the vectorial sum $K+L=\{x\in\rn\ :\ x=k+l, k\in K, l\in L \}$.
Many results have been proved regarding the volume of Minkowski sum of convex sets. 
In particular the volume of a convex combination of $m$ convex sets $K_1,...,K_m$ with weight $\la_1,...,\la_m$ is a polynomial of degree $m$ in the coefficients $\la_i$, as shown in the following proposition.
\begin{prop}[\cite{Sc} Theorem 5.1.7]
Let $K_1,...,K_m$ be convex sets in $\KK$ and $\la_1,...,\la_m$ non negative numbers.
There exists a non-negative symmetric function $V:(\KK)^N\to \RR$ such that
\begin{equation}
|\la_1K_1+...+\la_mK_m|=\sum_{i_1,...,i_N=1}^m \la_{i_1}\cdots\la_{i_N}V(K_{i_1},...,K_{i_N}).
\end{equation}
\end{prop}
The coefficients $V(K_{i_1},...,K_{i_N})$ (symmetric in their arguments) are named \emph{mixed volumes} of $K_1,...,K_m$.

For a regular (say at least $C^1$) convex set $K$, the {\em Gauss map} of  $K$, which associates to every point $x\in\partial K$ the outer unit normal vector of $\partial K$ at $x$, is denoted by $\nu(K,x):\pa K \to \Sbb^{N-1}$. If $K$ is strictly convex, $\nu$ is invertible and its inverse map is denoted by 
$\tau(K,\cdot)$ and in fact it coincides with the restriction of $D_\theta h(K,\cdot)$ at $\Sbb^{N-1}$ (see \cite{Sc}), that is
\begin{equation}\label{tauh}
\tau(K,\cdot)=\nu^{-1}(K,\cdot)=D_\theta h(K,\cdot):\Sbb^{N-1}\to\pa K\,,
\end{equation}
where we recall that $h(K,\theta)$ indicates the support function of the set $K$ in the direction $\theta$.
For an arbitrary convex domain $K$, $\tau(K,w)$ indicates the set of all boundary points of $K$ at which there exists a normal vector of $K$ belonging to the set $w\subset S^{N-1}$.

The Hausdorff measure of $\tau(K,\cdot)$ is called Area measure: 
$$
\textsf{S}_{N-1}(K,\theta)=\haus(\tau(K,\theta)).
$$
As for the volume, considering the area measure for a Minkowski combination of convex sets leads to the notion of mixed area measures, as shown in the following proposition.
\begin{prop}[\cite{Sc} Theorem 5.1.6]
Let $K_1,...,K_m$ be convex sets in $\KK$ and $\la_1,...,\la_m$ non negative numbers.
There exists a symmetric map $\textsf{S}$ from $(\KK)^N$ to the space of finite Borel measures on the sphere $\Sbb^{N-1}$ such that
\begin{equation}\label{SNmeno1}
{\textsf{S}_{N-1}}(\la_1K_1+...+\la_mK_m, \cdot)=\sum_{i_1,...,i_{N-1}=1}^m \la_{i_1}\cdots\la_{i_{N-1}}\textsf{S}(K_{i_1},...,K_{i_{N-1}}, \cdot).
\end{equation}
\end{prop}
The coefficients $\textsf{S}(K_{i_1},...,K_{i_{N-1}}, \cdot)$ (symmetric in their arguments) are called {\em mixed area measures} of $K_1,...,K_m$.

Mixed volumes and mixed area measures are related by the following integral formula (see \cite[Theorem 5.1.7]{Sc}):
\begin{equation}\label{integralMixedVolume}
V(K_1,K_2,...,K_N)=\frac 1N\int_{\Sbb^{N-1}}h(K_1,\theta)\textsf{S}(K_2,...,K_N,d\theta).
\end{equation}

One of the most important results on mixed volume is a system of quadratic inequalities called Aleksandrov-Fenchel inequalities, satisfied by general mixed volumes. 
A special version is the following: (see \cite[Section 7.3]{Sc})
for $K,L\in\KK$ it holds that
\begin{equation}\label{alekIneq}
V(L,K,...,K)^2 \ge |K|\; V(L,L,K,...,K).
\end{equation}

In the proof of our main result, we will use \eqref{alekIneq} when an anisotropic ball and a general convex set are considered (i.e. $L=\bho, K=\Om$).
\begin{prop}\label{mink-ineq-prop}
Let $H$ be a norm of $\rn$ of class $C^2(\rn \setminus \{O\})$ and let $\Om$ be a regular bounded convex domain in $\rn$; it holds:
\begin{equation}\label{mink-ineq}
P_H^2(\Om) \ge N|\Om|\,\int_{\pa\Om}\frac{\Hcal(x)}{N-1}H(\nu(x))\;d\haus(x).
\end{equation}
\end{prop}
\begin{rem}
Notice that (\ref{mink-ineq}) can be seen as the anisotropic version of the Minkowski inequality (\cite[Theorem 7.2.1]{Sc})
$$
\Big(\haus(\pa\Om)\Big)^2\ge N|\Om|\int_{\pa\Om}H_1\;d\haus(x),
$$
where $H_1$ denotes the standard mean curvature of $\pa\Om$.
\end{rem}
\begin{proof}[Proof of Proposition \ref{mink-ineq-prop}]
We first prove the theorem for strictly convex domains, then  the general statement follows by approximation.

Let $\Omega$ be a strictly convex domain; we compute the mixed volumes in the Aleksandrov-Fenchel inequality \eqref{alekIneq} for $L=\bho$ and $K=\Om$.
We first show that:
\begin{equation}\label{V(BKK)}
V(\bho,\Om,...,\Om)= \frac 1N P_H(\Om).
\end{equation}
Indeed, thanks to the integral formula \eqref{integralMixedVolume}, the fact that $\textsf{S}(\Om,...,\Om,\theta)=\haus(\tau(\Om,\theta))$ and \eqref{HhB}, we can compute
\begin{eqnarray*}
V(\bho,\Om,...,\Om)&=&\frac 1N\int_{\Sbb^{N-1}}h(\bho,\theta)\;d\haus(\tau(\Om,\theta))\\
	&=&\frac 1N\int_{\pa\Om}h(\bho,\nu(\Om,x))\;d\haus(x)=\frac 1N\int_{\pa\Om}H(\nu(\Om,x))\;d\haus(x)\\
	&=&\frac 1NP_H(\Om).
\end{eqnarray*}
Notice that the function $\tau(\Om,\cdot)$ is well defined since $\Om$ is strictly convex by assumption.

Let us now show that
\begin{equation}\label{paolo1}
V(\bho,\bho,\Om,...,\Om)=\frac 1N \int_{\pa\Om}\frac{\Hcal(x)}{N-1} H(\nu(x))\;d\haus(x).
\end{equation}
Indeed, by \eqref{integralMixedVolume} and (\ref{HhB}) we have
\begin{equation}\label{paolo2}
V(\bho,\bho,\Om,...,\Om)=\frac 1N \int_{S^{N-1}}H(\theta)\,\textsf{S}(\bho,\Om,...,\Om,\theta)\,d\theta\,.
\end{equation}
Let $\theta\in\Sbb^{N-1}$ be fixed and choose an orthonormal basis $(e_1,\dots,e_N)$ or $\rn$ with $e_N=\theta$. Then it holds (see \cite[(2.68)]{Sc}) that
\begin{equation}\label{SKtheta}
\textsf{S}(K_1,K_2,\dots,K_{N-1},\theta)=D\big((h_{ij}(K_1,\theta))_{i,j=1}^{N-1},\dots,(h_{ij}(K_{N-1},\theta))_{i,j=1}^{N-1}\big)\,,
\end{equation}
where $D(A_1,\dots,A_m)$ denotes the \emph{mixed discriminant} of the $(N-1)\times(N-1)$ matrices $A_1,\dots,A_m$; the mixed discriminants are  symmetric in their arguments and then uniquely determined by the formula
\begin{equation}\label{defmixedD}
\det(\lambda_1A_1+\dots+\lambda_mA_m)=\sum_{i_1,\dots,i_k=1}^m\lambda_{i_1}\cdots\lambda_{i_k}D(A_{i_1},\dots,A_{i_k})\,.
\end{equation}

Then, by \eqref{SKtheta}, we have
\begin{equation}\label{p1}
\textsf{S}(\bho,\Om,...,\Om,\theta)=D\big((H_{ij}(\theta))_{i,j=1}^{N-1},(h_{ij}(\Omega,\theta))_{i,j=1}^{N-1},\dots,(h_{ij}(\Omega,\theta))_{i,j=1}^{N-1}\big)
\end{equation}
and by \eqref{defmixedD}
\begin{equation*}\label{p2}
D\big((H_{ij}(\theta))_{i,j=1}^{N-1},(h_{ij}(\Omega,\theta))_{i,j=1}^{N-1},\dots,(h_{ij}(\Omega,\theta))_{i,j=1}^{N-1}\big)
=\frac{1}{(N-1)}\,\frac{d}{dt}\det(A_1+tA_2)_{|t=0}\,,
\end{equation*}
where $A_1=(h_{ij}(\Omega,\theta))_{i,j=1}^{N-1}$ and
$A_2=(H_{ij}(\theta))_{i,j=1}^{N-1}$ (and $m=2$, $\lambda_1=1$, $\lambda_2=t$).

From \eqref{derivdet}, we get
\begin{equation*}\label{paolo3}
\frac{d}{dt}\det(A_1+tA_2)_{|t=0}=\sum H_{ij}(\theta) \textsf{S}_{N-1}^{ij}((h_{rs}(\Omega,\theta))_{r,s=1}^{N-1})
\end{equation*}
and  \eqref{SNij} tells
\begin{equation}\label{questasi}
\textsf{S}_{N-1}^{ij}((h_{rs}(\Omega,\theta))_{r,s=1}^{N-1})=h^{ij}\,\det((h_{rs}(\Omega,\theta))_{r,s=1}^{N-1})\,,
\end{equation}
where $(h^{ij})_{i,j=1}^{N-1}$ denotes the inverse matrix of $(h_{ij}(\Omega,\theta))_{i,j=1}^{N-1}$.

The chain of equalities from \eqref{p1} to \eqref{questasi} yields
$$
\textsf{S}(\bho,\Om,...,\Om,\theta)=\frac1{N-1}\sum H_{ij}(\theta)h^{ij}\det((h_{rs}(\Omega,\theta))_{r,s=1}^{N-1})\,,
$$
and, since \eqref{tauh} gives $h^{ij}(\Omega,\theta)=\nu_i^j(\Omega,\tau(\Omega,\theta))$, we obtain
$$
\textsf{S}(\bho,\Om,...,\Om,\theta)=\frac1{N-1}\sum H_{ij}(\theta)\nu_i^j\det((h_{rs}(\Omega,\theta))_{r,s=1}^{N-1})\,.
$$
Finally, from \eqref{MHOm} we get
\begin{equation}\label{S=M}
\textsf{S}(\bho,\Om,...,\Om;\theta)=\frac{\Hcal(\tau(\Omega,\theta))}{N-1}\det((h_{rs}(\Omega,\theta))_{r,s=1}^{N-1}) \,.
\end{equation}
Inserting the latter into \eqref{paolo2}, using the change of variable $\theta=\nu(\Omega,x)$ (equivalently $x=\tau(\Omega,\theta)$) and taking into account \eqref{tauh}, we obtain
\eqref{paolo1}.
Coupling \eqref{paolo1}, \eqref{V(BKK)} and \eqref{alekIneq}, we get  \eqref{mink-ineq}.
\end{proof}


\section{Proof of Theorem \ref{theorem Omega ball}} \label{sec_mainresult}

We consider the auxiliary function $v(x)=u(x)^{\frac2{N-2}}$ which solves the following problem
\begin{equation}\label{Pbv}
\begin{cases}
\Delta_H v=\dfrac Nv V(Dv) \qquad&\text{ in }\rn\setminus\overline{\Om},\\
v=1 \qquad&\text{ on }\pa\Om,\\
H(Dv)= \frac2{N-2}C \qquad&\text{ on }\pa\Om,\\
v\to +\infty \qquad&\text{ if }|x|\to\infty,
\end{cases}
\end{equation}
where $V(\xi)=\frac 12H^2(\xi)$ and $C$ is the same constant as in \eqref{3bis}.

We define the matrix $W(x)=\W{x}$ whose elements  are denoted by $w_{ij}$.
In order to simplify the presentation, arguments are omitted and hence $H, V$ denote $H(Dv(x)), V(Dv(x))$, respectively.

Since $\sdue{W}$ is divergence free (see (\ref{divS2=0})), we have that
\begin{equation*}
\diver(v^{1-N}\sdue{W}V_{\xi_i}) = (1-N)v^{-N}\sdue{W}V_{\xi_i}v_j+v^{1-N}\sdue{W}V_{\xi_i\xi_k}v_{kj} \,,
\end{equation*}
and from \eqref{eq200} we obtain that
\begin{equation} \label{divvsdue}
\diver(v^{1-N}\sdue{W}V_{\xi_i}) = (1-N)v^{-N}\sdue{W}V_{\xi_i}v_j+2v^{1-N}S_2(W) \,.
\end{equation}
From \eqref{S2ijW} we can write
$$
\sdue{W}V_{\xi_i}v_j = -w_{ji}V_{\xi_i}v_j+\Delta_Hv\,V_{\xi_i}v_i
$$
so that (\ref{Hlaplacian-curvatura}) yields
\begin{equation} \label{sdueVivj}
\sdue{W}V_{\xi_i}v_j =-V_{\xi_j\xi_k}v_{ki}V_{\xi_i}v_j + 2V(H_{\xi_i}H_{\xi_j}v_{ij}+\Hcal(\Om)H) \,.
\end{equation}
Moreover, recalling that $V=\frac 12H^2$ and the homogeneity properties (\ref{nablaH.xi}), (\ref{D2Hxi}) of $H$, it holds 
\begin{eqnarray*}
V_{\xi_j\xi_k}v_{ki}V_{\xi_i}v_j= HH_{\xi_k}H_{\xi_i}v_{ki}\,H_{\xi_j}v_j+H^2H_{\xi_k\xi_j}H_{\xi_i}v_{ki}v_j=H^2H_{\xi_i}H_{\xi_k}v_{ki}.
\end{eqnarray*}
By coupling this latter with relation (\ref{sdueVivj}) we get
\begin{equation}\label{sdueVivj2}
\sdue{W}V_{\xi_i}v_j = H^3\Hcal(\Om) \,.
\end{equation} 

By \eqref{sdueVivj} and the homogeneity properties (\ref{nablaH.xi}), (\ref{D2Hxi}) it holds that 
\begin{eqnarray*}
v^{-N}\sdue{W}V_{\xi_i}v_j=2v^{-N}V\Delta_Hv-v^{-N}V_{\xi_i}V_{\xi_k}v_{ki}.
\end{eqnarray*}
Moreover, using the fact that $v^{-N}V_{\xi_i}V_{\xi_k}v_{ki}=2NV^2v^{-(N+1)}+\diver(v^{-N}V\nxi V)$ and the first equation (\ref{Pbv}), we have
\begin{equation}\label{vnsdue}
v^{-N}\sdue{W}V_{\xi_i}v_j = Nv^{-(N+1)}V^2-\diver(v^{-N}V\nxi V) \,,
\end{equation}
for every $x\in\rn\setminus\overline{\Om}$.

Let $R$ be large so that $\Om$ is contained in the Euclidean ball $B^R$ with radius $R$ and centered at the origin.
We are going to compute 
$$
I=\int_{B^R\setminus\overline{\Om}}v^{-N}\sdue{W}V_{\xi_i}v_j \;dx
$$ 
by using expressions (\ref{divvsdue}) from one hand and (\ref{vnsdue}) on the other hand.

From (\ref{divvsdue}), the Divergence Theorem and the fact that $\nu=-Dv/|Dv|$, we compute
\begin{eqnarray*}
I &=& 
	      -\frac 1{N-1}\int_{B^R\setminus\overline{\Om}} \diver(v^{1-N}\sdue{W}V_{\xi_i})\;dx+\frac2{N-1}\int_{B^R\setminus\overline{\Om}}v^{1-N}S_2(W)\;dx\\
	&=& \frac 1{N-1} \int_{\pa\Om}v^{1-N}\sdue{W}V_{\xi_i}\frac{v_j}{H(Dv)}H(\nu)\;d\haus(x)-\frac1R\frac1{N-1}\int_{\{|x|=R\}}v^{1-N}\sdue{W}V_{\xi_i}x_j\;\\
	&& \hspace{7cm}{ + \frac2{N-1}\int_{B^R\setminus\overline{\Om}}v^{1-N}S_2({W})\;dx.}
\end{eqnarray*}
Notice that the first term in the latter expression can be rewritten by using the boundary conditions in (\ref{Pbv}) and equation (\ref{sdueVivj2}) as
$$
\frac 1{N-1} \int_{\pa\Om}v^{1-N}\sdue{W}V_{\xi_i}\frac{v_j}{H(Dv)}H(\nu)\;d\haus(x)=\frac{4C^2}{(N-2)^2}\int_{\pa\Om} \frac{\Hcal(\Om)}{N-1} H(\nu)\;d\haus(x),
$$
and hence it holds
\begin{eqnarray} \label{I1}
I&=& \frac{4C^2}{(N-2)^2}\int_{\pa\Om} \frac{\Hcal(\Om)}{N-1} H(\nu)\;d\haus(x)-\frac1R\frac1{N-1}\int_{\{|x|=R\}}v^{1-N}\sdue{W}V_{\xi_i}x_j \nonumber\\
	&& \hspace{6cm}{ + \frac2{N-1}\int_{B^R\setminus\overline{\Om}}v^{1-N}S_2({W})\;dx.}
\end{eqnarray}

On the other hand the value of $I$ can be computed by using (\ref{vnsdue}), the Divergence Theorem and the fact that $\nu=-Dv/|Dv|$, as follows:
\begin{eqnarray*}
I &=&
	\int_{B^R\setminus\overline{\Om}} NV^2v^{-(N+1)}\;dx-\int_{B^R\setminus\overline{\Om}}\diver(v^{-N}V\nxi V)\;dx\\
	&=& \int_{B^R\setminus\overline{\Om}} NV^2v^{-(N+1)}\;dx- \frac 1R\int_{\{|x|=R\}}v^{-N}VV_{\xi_i}x_i\;+\int_{\pa{\Om}}v^{-N}VV_{\xi_i}\frac{v_i}{H(Dv)}H(\nu)\;d\haus(x).
\end{eqnarray*}
The last term can be rewritten by using the boundary conditions in (\ref{Pbv}) in the following way:
$$
\int_{\pa{\Om}}v^{-N}VV_{\xi_i}\frac{v_i}{H(Dv)}H(\nu)\;d\haus(x)=\frac12\Big(\frac{2C}{N-2}\Big)^3\int_{\pa\Om}H(\nu)\;d\haus(x),
$$
which gives
\begin{equation}\label{I2}
I = \int_{B^R\setminus\overline{\Om}} NV^2v^{-(N+1)}\;dx- \frac 1R\int_{\{|x|=R\}}v^{-N}VV_{\xi_i}x_i\;dx+\frac12\Big(\frac{2C}{N-2}\Big)^3 P_H(\Om) \,,
\end{equation}
where we used \eqref{anis_surf_energy}.

Notice that Theorem \ref{thm properties u} implies
\begin{eqnarray*}
&&\lim_{R\to+\infty}\frac1R\frac1{N-1}\int_{\{|x|=R\}}v^{1-N}\sdue{W}V_{\xi_i}x_j\;dx=0\\
&&\lim_{R\to+\infty} \frac 1R\int_{\{|x|=R\}}v^{-N}VV_{\xi_i}x_i\;dx =0.
\end{eqnarray*}
Hence, passing to the limit $R\to+\infty$ in (\ref{I1}) and (\ref{I2}) and coupling them, we find that
\begin{eqnarray} \label{I1=I2}
&&\frac{4C^2}{(N-2)^2}\int_{\pa\Om} \frac{\Hcal(\Om)}{N-1} H(\nu)\;d\haus(x) + \frac2{N-1}\int_{\rn\setminus\overline{\Om}}v^{1-N}S_2({W})\;dx\\
	&& \hspace{2cm}=\int_{\rn\setminus\overline{\Om}} NV^2v^{-N-1}\;dx+\frac12\Big(\frac{2C}{N-2}\Big)^3 P_H(\Om).
\end{eqnarray}

Since $W$ is the product of symmetric matrices, with $\nxi^2V$ positive semidefinite, Newton's Inequality (\ref{newtonIneq}) holds for $W$.
Plugging it in (\ref{I1=I2}) we obtain
\begin{equation} \label{totti}
P_H(\Om)\le \frac{N-2}{C} \int_{\pa\Om} \frac{\Hcal(\Om)}{N-1} H(\nu)\;d\haus(x).
\end{equation}
We use the value of $C$ in (\ref{C=}) and notice that \eqref{totti} is the reverse inequality of (\ref{mink-ineq}); then equality must hold.
Hence the equality sign in (\ref{newtonIneq}) holds true too, which implies that $W$ is a multiple of the identity matrix $Id$, that is
\begin{equation}\label{W=gammaI}
\W{x}=\gamma(x)\;Id,
\end{equation}
for every $x\in\rn\setminus\overline{\Omega}$.

Recalling that $\Delta_Hv=\tr(W)$ and expression (\ref{Hlaplacian-curvatura}), the latter entails
\begin{equation}\label{Hcal=gammaX}
\Hcal(S_t) H(Dv)+H_{\xi_i}H_{\xi_j}v_{ij}= N\gamma,\qquad x\in\rn\setminus\overline{\Omega}.
\end{equation}
Moreover (\ref{W=gammaI}) and \eqref{miracolo} imply that 
$$
v_{ij}=\gamma \frac{\partial^2}{\partial \eta_i\eta_j}V_0(\nabla_{\xi}H(Dv)),
$$
where $V_0(\eta)$ is the dual function of $V(\xi)$, that is $V_0(\eta)=\frac 12H_0^2$.
Hence the following holds
$$
H_{\xi_i}H_{\xi_j}v_{ij}=\gamma \frac{\partial^2}{\partial \eta_i\eta_j}V_0 (\nabla_{\xi}H(Dv))H_{\xi_i}H_{\xi_j}=2\gamma V_0(\nabla_\xi H)=\gamma,
$$
thanks to the homogeneity property (\ref{nablaH.xi}) of $V_0$ and (\ref{H0H=1}).
Then (\ref{Hcal=gammaX}) can then be rewritten as
\begin{equation}\label{Hcal=gamma}
\Hcal(S_t) H(Dv)+\gamma(x)= N\gamma(x) \,,
\end{equation}
for every $x\in\rn\setminus\overline{\Omega}$.

Notice that, {thanks to the regularity result in Theorem \ref{thm properties u}} the function $\gamma$ is constant on  $\pa\Omega$ since 
$$
N\gamma=\frac Nv V(Dv) = \frac N2 \left(\frac{2}{N-2} C \right)^2 \,,
$$ 
from (\ref{Pbv}).
Hence relation (\ref{Hcal=gamma}) implies that $\Hcal(\Omega)$ is constant, being $\Omega$ the level set $S_1$ of $v$.

The proof is concluded thanks to Theorem \ref{AleksandrovThm} which assures that $\Omega$ is a ball in the $H_0$ norm: $\Omega=\bhor{r}$.

\appendix

\section{Proof of \eqref{C=}}\label{appendixC}
\begin{prop}
If there exists a solution {$u \in C^{2,\alpha}(\rn \setminus \Om)$} of \eqref{pb_capacity}-\eqref{3bis}, then  $C= \frac{N-2}N \frac{P_H(\Om)}{|\Om|}$.
\end{prop}
\begin{proof}
\emph{First step:} $\Hcap(\Om) = C P_H(\Om).$

Since $H$ is 1-homogeneous, $\nu=Du/|Du|$, and from \eqref{3bis}, it holds that
\begin{eqnarray*} 
C\; P_H(\Om) &=& C\int_{\pa\Om}H(\nu)\; d\haus(x)=C\int_{\pa\Om}H(\frac{Du}{|Du|})\; d\haus(x)\\
	&=& \int_{\pa\Om} \frac1{|Du|}H(Du)H(Du)\; d\haus(x) \,,
\end{eqnarray*}
so that
\begin{equation} \label{CP_H}
C\; P_H(\Om) = \int_{\pa\Om} \frac1{|Du|}H^2(Du)\; d\haus(x) \,.
\end{equation}

On the other hand, by using coarea formula and \eqref{nablaH.xi}, we find that
\begin{eqnarray*}
\Hcap(\Om) &=& \int_{\rn\setminus\overline{\Om}}H^2(Du)\;dx= \int_0^1\int_{\{u=t\}}\frac1{|Du|} H^2(Du)\; d\haus(x)\\
	&=& \int_0^1 \int_{\{u=t\}} H(Du) \langle \nxi H(Du) ;  \frac{Du}{|Du|} \rangle  \; d\haus(x) \\
	&=& -\int_0^1 \int_{\{u=t\}} H(Du) \langle \nxi H(Du) ;  \nu \rangle \; d\haus(x).
\end{eqnarray*}
Notice that, since $\Delta_H u=0$ in $\{u\ge t\}\setminus\Om$, by using Divergence Theorem we have that the quantity
$$
\int_{\{u=t\}}  H(Du) \langle \nxi H(Du) ;  \nu \rangle \; d\haus(x)
$$
is independent of the level $t\in(0,1]$.
Hence it holds
$$
\Hcap(\Om) = -\int_0^1 \int_{\{u=1\}} H(Du) \langle \nxi H(Du) ; \nu \rangle \; d\haus(x) =\int_{\pa\Om}\frac{H^2(Du)}{|Du|} \; d\haus(x),
$$
which entails, together with \eqref{CP_H}, that $C\; P_H(\Om)=\Hcap(\Om)$.

\emph{Second step:} $(N-2) \Hcap(\Om)=C^2 N |\Om|$.

By the Divergence Theorem and \eqref{3bis}, we compute 
\begin{eqnarray*}
C^2 N|\Om| &=& \int_{\pa\Om} \langle x ; \nu \rangle \, H^2(Du) \; d\haus(x)= -\int_{\rn\setminus\overline{\Om}}\diver(H^2(Du)x)\;dx\\
	&=& -\int_{\rn\setminus\overline{\Om}}NH^2(Du)\;dx-\int_{\rn\setminus\overline{\Om}} 2V_i(Du)x_j u_{ij}\;dx\\
	&=& -N\Hcap(\Om)-2\left( \int_{\rn\setminus\overline{\Om}} \langle \nxi V(Du) ; D^2u\, x +Du \rangle \; dx - 2 \int_{\rn\setminus\overline{\Om}} 2 V(Du)\;dx\right),
\end{eqnarray*}
where $V=H^2/2$ and the last equality holds thanks to the homogeneity of $V(\cdot)$, which follows from \eqref{nablaH.xi}.

Recalling the definition of H-capacity, the fact that $\diver(\nxi V(Du))=0$ in $\rn\setminus\overline{\Om}$ and that $\nu=-Du/|Du|$ on $\pa\Om$, and by using the homogeneity of $V$, the latter can be rewritten as
\begin{eqnarray*}
C^2 N|\Om| &=& -N\Hcap(\Om)-2 \int_{\rn\setminus\overline{\Om}} \langle \nxi V(Du) ; D(Du\cdot x) \rangle \; dx + 2\Hcap(\Om)\\
	&=& (2-N)\Hcap(\Om)- 2 \int_{\rn\setminus\overline{\Om}} \diver(\nxi V(Du)(Du\cdot x))\;dx\\
	&=& (2-N)\Hcap(\Om) + 2 \int_{\pa\Om} \langle Du ;  x \rangle  \langle \nxi V(Du) ;  \nu \rangle \;d\haus(x)\\
	&=& (2-N)\Hcap(\Om) - 2 \int_{\pa\Om} \langle \frac{Du}{|Du|} ; x\rangle  \langle \nxi V(Du) ;  Du \rangle \;d\haus(x)\\
	&=& (2-N)\Hcap(\Om)+ 2 \int_{\pa\Om} (\nu\cdot x) H^2(Du)\;d\haus(x)\\
	&=& (2-N)\Hcap(\Om)+ 2 C^2 \int_{\pa\Om} \nu\cdot x\;d\haus(x)= (2-N) \Hcap(\Om)+2C^2N|\Om| \,,
\end{eqnarray*}
which completes Step 2. The desired expression of $C$ is achieved by coupling the two steps.
\end{proof}
\section{Proof of Theorem \ref{AleksandrovThm}}\label{appendixB}
	Let $\psi$ be the solution to the following problem:
	\begin{equation}\label{pbAppendix}
	\begin{cases}
	\Delta_H \psi=1&\qquad\text{ in }\Omega,\\
	\psi=0&\qquad\text{ on }\partial\Omega,
	\end{cases}
	\end{equation}
	and let $W=\nabla^2_{\xi}V(D\psi)D^2\psi$, where $V(\xi)=\frac 12 H^2(\xi)$.
	Notice that, thanks to (\ref{divS2=0}), it holds
	$$
	S_2(W)= \frac 12 \diver(\sdue{W}V_{\xi_i}). 
	$$
	The latter, together with Newton's Inequality (\ref{newtonIneq}) and (\ref{pbAppendix}), implies
$$
\frac{N-1}{2N} = \frac{N-1}{2N}(\Delta_H \psi)^2 \ge \frac 12 \diver(\sdue{W}V_{\xi_i}) \,,
$$
for every $x\in\Omega$.
By integrating over $\Omega$, using the Divergence Theorem on the right hand side, using that $\nu=Du/|Du|$ and  (\ref{sdueVivj2}), one obtains that
\begin{eqnarray*}
\frac{N-1}{2N} |\Omega| &\ge& \frac 12 \int_{\partial\Omega}\sdue{W}V_{\xi_i}(D\psi)\nu_j= \frac12 \int_{\partial\Omega}\sdue{W}V_{\xi_i}(D\psi) \frac{\psi_j}{|D\psi|}\\
&=&\frac12 \int_{\partial\Omega}\frac{H^3(D\psi)}{|D\psi|}\Hcal(\Omega) \,,
\end{eqnarray*}
and, since $H$ is 1-homogenous, we find
\begin{equation} \label{appendix4}
\frac{N-1}{2N} |\Omega| \geq \frac 12 \Hcal(\Omega)\int_{\partial\Omega}H(\nu)H^2(D\psi) \,.
\end{equation}
On the other hand by Cauchy-Schwarz inequality it holds
\begin{equation}\label{appendix1}
\left(\int_{\partial\Omega}H(\nu)H(D\psi)\right)^2\le \int_{\partial\Omega}H(\nu)\int_{\partial\Omega}H(\nu)H^2(D\psi),
\end{equation}
and hence, by the definition of $\Delta_H\psi$ and \eqref{nablaH.xi}, we  obtain
$$
|\Omega|=\int_\Omega \Delta_H\psi =\int_\Omega \langle \nabla_\xi V(D\psi); \nu\rangle  = \int_\Omega \frac{H(D\psi)}{|D\psi|}\langle\nabla_\xi H(D\psi); D\psi\rangle \,,
$$
so that
\begin{equation} \label{appendix2}
|\Omega| = \int_\Omega H(\nu)H(D\psi) \,.
\end{equation}

Coupling (\ref{appendix1}) and (\ref{appendix2}) we obtain
$$
|\Omega|^2\le \int_{\partial\Omega}H(\nu)\int_{\partial\Omega}H(\nu)H^2(D\psi).
$$	
Recalling inequality (\ref{appendix4}) and definition (\ref{anis_surf_energy}) we have proved that
\begin{equation}\label{appendix3}
|\Omega|^2\le P_H(\Omega)\frac{|\Omega|}{N}\frac{(N-1)}{\Hcal(\Omega)}.
\end{equation}
Thanks to the anisotropic Minkowski type formula 
$$
P(\Omega)=\int_{\partial\Omega}\frac{\Hcal(\Omega)}{N-1}\langle x; \nu\rangle
$$
(see \cite{HeLi}) and the fact that $\Hcal(\Omega)$ is constant, it holds that
$$
P(\Omega)= N\frac{\Hcal(\Omega)}{N-1}|\Omega| \,,
$$
and hence the equality sign must hold in (\ref{appendix3}).
This entails that equality holds in both Newton and Cauchy-Schwarz inequalities and hence $H(D\psi)$ must be constant on $\partial\Omega$, that is: the overdetermined anisotropic Serrin Problem must be satisfied. Thanks to \cite[Theorem 2.2]{CS}, the set $\Omega$ is then Wulff shape.
%


\section*{Acknowledgements}
The work has been supported by the FIR project 2013 ``Geometrical and Qualitative aspects of PDE'' and the GNAMPA of the Istituto Nazionale di Alta Matematica (INdAM). Part of this work has been done while the second author was visiting the University of Texas at Austin under the support of NSF-DMS FRG Grant 1361122, of a Oden Fellowship at ICES.

\medskip

Conflict of Interest: The authors declare that they have no conflict of interest.


\end{document}